\documentclass[11pt,a4paper]{article}
\usepackage{amsmath}
\usepackage{amssymb}
\usepackage{amsthm}
\usepackage{amsfonts}
\usepackage{enumerate}
\usepackage{pstricks}
\usepackage{pst-plot}
\usepackage{pst-poly}
\usepackage{url}
\usepackage{booktabs}
\usepackage{verbatim}
\usepackage{hyperref}
\usepackage{breakurl,url}
\usepackage[margin=1.3in]{geometry}
\def\tluste#1{\protect{\textrm{\boldmath $#1$}}}
\newcommand{\vr}[1]{{{#1}}}

\newcommand{\mace}[1]{{{#1}}}
\newcommand{\mna}[1]{{\mathcal{#1}}}

\newcommand{\omace}[1]{\mbox{$\overline{\mace{#1}}$}} 
\newcommand{\umace}[1]{\mbox{$\underline{\mace{#1}}$}} 
\newcommand{\imace}[1]{\mbox{$\tluste{#1}$}}

\def\Mid#1{#1^c}
\def\Rad#1{#1^\Delta}
\newcommand{\ovr}[1]{\mbox{$\overline{\vr{#1}}$}} 
\newcommand{\uvr}[1]{\mbox{$\underline{\vr{#1}}$}}

\newcommand{\onum}[1]{\mbox{$\overline{{#1}}$}} 
\newcommand{\unum}[1]{\mbox{$\underline{{#1}}$}}

\newcommand{\ivr}[1]{\tluste{#1}} 		

\newcommand{\R}[0]{{\mathbb{R}}}

\newcommand{\IR}[0]{{\mathbb{IR}}}

\newcommand{\mmid}[0]{;\,}		

\newcommand{\seznam}[1]{{\{1, \ldots, {#1}\}}}

\def\clqq{``}
\def\crqq{''}
\def\quo#1{\clqq{}#1\crqq{}}  

\newcommand{\sts}[0]{\mathop{\mbox{subject to}}}
\newcommand{\stl}[0]{\sts\ \ }
\newcommand{\st}[0]{\ \ \sts\ \ }

\DeclareMathOperator{\sgn}{sgn}	
\DeclareMathOperator{\diag}{diag}	

\def\nref#1{\eqref{#1}}

\newtheorem{proposition}{Proposition}
\newtheorem{lemma}{Lemma}
\newtheorem{corollary}{Corollary}

\theoremstyle{definition}

\newtheorem{example}{Example}
\newtheorem{remark}{Remark}

\begin{document}

\title{Linear programming sensitivity measured by the optimal value worst-case analysis}

\author{
  Milan Hlad\'{i}k\footnote{
Charles University, Faculty  of  Mathematics  and  Physics,
Department of Applied Mathematics, 
Malostransk\'e n\'am.~25, 11800, Prague, Czech Republic, 
e-mail: \texttt{milan.hladik@matfyz.cz}
}}

\date{\today}
\maketitle

\begin{abstract}
This paper introduces a concept of a derivative of the optimal value function in linear programming (LP). Basically, it is the the worst case optimal value of an interval LP problem when the nominal data the data are inflated to intervals according to given perturbation patterns. 
By definition, the derivative expresses how the optimal value can worsen when the data are subject to variation. In addition, it also gives a certain sensitivity measure or condition number of an LP problem.

If the LP problem is nondegenerate, the derivatives are easy to calculate from the computed primal and dual optimal solutions. For degenerate problems, the computation is more difficult. We propose an upper bound and some kind of characterization, but there are many  open problems remaining.

We carried out numerical experiments with specific LP problems and with real LP data from Netlib repository. They show that the derivatives give a suitable sensitivity measure of LP problems. It remains an open problem how to efficiently and rigorously handle degenerate problems.
\end{abstract}


\section{Introduction}

Stability and sensitivity are important topics of investigation in linear programming (LP). 
They study the question how the optimal solution and optimal value change in a local neighborhood of the data; the sensitivity analysis often focuses more on the change in a specific direction of the data variations on a certain stability region \cite{ArsObl1990,CurLee2022,Gal1995,GalGre1997,HadTer2006b,LuBor2022p,Luc2012,MosSol2014}.
The related areas of parametric programming \cite{Gal1995,GalGre1997,Hla2010b,MehHad2021,MehHad2022p} and tolerance analysis \cite{BorBuz2018,Fil2005,Hla2011c,WarWen1990,Wen1997} study possibly larger and independent perturbations of LP coefficients.

A condition number measures how much errors or other changes in data affect the results of a calculation. This is an important concept in numerical linear algebra \cite{GolLoa1996,Hig1996}, but the idea has been used in other algebraic and optimization problems \cite{CouCro2009,SonXu2015,ZamHla2020ab}, too. Particularly in linear programming, it quantitatively expresses a stability of an LP problem. The early works are by Mangasarian \cite{Mang1981}, Renegar~\cite{Ren1994,Ren1995} and others afterward. A unification of some of the concepts was presented by Cheung et al.~\cite{CheuCuc2003}. Computing a condition number is usually a computationally demanding problem, however, Freund and Vera~\cite{FreVer2003} presented some polynomially solvable cases emerging when a suitable matrix norm is used.

This paper introduces the concept of a worst case derivative of the optimal value in linear programming. Thus, the value measures the worst rate of change of the optimal value in a small neighborhood. It can also serve as a kind of condition number since it expresses a local sensitivity of an LP problem.

Throughout the paper (with an explicitly stated exception of Section~\ref{ssOtherLP}), we consider an LP problem in the form
\begin{align}\label{lpGen}
f(A,b,c)=\min\ c^Tx \st Ax=b,\ x\geq0,
\end{align}
where $A\in\R^{m\times n}$, $b\in\R^m$ and $c\in\R^n$. 
For an index set $B\subseteq\seznam{n}$, we denote by $A_B$ the restriction of $A$ to the columns indexed by~$B$; when $A_B$ is nonsingular, then $B$ is called a basis. The nonbasic indices are denoted by~$N$.

The other notation used is as follows. 
The sign of a real $r$ is defined $\sgn(r)=1$ if $r\geq0$ and $\sgn(r)=-1$ otherwise; for vectors we apply it entrywise. 
Further, $I_n$ is the identity matrix of size $n\times n$, $e$ is the vector of ones with convenient dimension, and $\diag(v)$ is the diagonal matrix with entries given by the vector $v=(v_1,\dots,v_n)^T$. The Frobenius norm of a matrix $M$ is denoted by $\|M\|_F$.

\paragraph{Interval data.}
Since our approach builds on the theory of interval analysis and interval programming, we have to remind some essential notation and results first. An interval matrix is defined as the set
$$
\imace{A}
=[\umace{A},\omace{A}]
=\{A\in\R^{m\times n}\mmid \umace{A}\leq A\leq\omace{A}\},
$$
where $\umace{A},\omace{A}\in\R^{m\times n}$, $\umace{A}\leq\omace{A}$ are given matrices. 
The midpoint and the radius matrix are defined, respectively, as 
$$
\Mid{A}:=\frac{1}{2}(\umace{A}+\omace{A}),\quad
\Rad{A}:=\frac{1}{2}(\omace{A}-\umace{A}).
$$
The set of all $m\times n$ interval matrices is denoted by $\IR^{m\times n}$. 
Similar notation is used for interval vectors, considered as one column interval matrices, and interval numbers. 
For interval arithmetic see, e.g., the textbooks \cite{HanWal2004,MooKea2009}.

\paragraph{Interval linear programming.}
Let $\imace{A}\in\IR^{m\times n}$, $\ivr{b}\in\IR^m$ and $\ivr{c}\in\IR^n$ be given. By \emph{an interval linear programming problem} we mean a family of LP problems \nref{lpGen} with $A\in\imace{A}$, $b\in\ivr{b}$ and $c\in\ivr{c}$. A particular LP problem from this family is called a \emph{realization}.
References \cite{Fie2006,Hla2012a} present a survey on this topic.

The best case and worst case optimal values are defined as
\begin{align*}
\unum{f}(\imace{A},\ivr{b},\ivr{c})
 &:=\min\ f(A,b,c) \st A\in\imace{A},\ b\in\ivr{b},\ c\in\ivr{c},\\
\onum{f}(\imace{A},\ivr{b},\ivr{c})
 &:=\max\ f(A,b,c) \st A\in\imace{A},\ b\in\ivr{b},\ c\in\ivr{c}.
\end{align*}
The best case optimal value is easily computable by solving a specific linear program
\begin{align*}
\unum{f}(\imace{A},\ivr{b},\ivr{c})&= \min\ \uvr{c}^Tx \st
 \umace{A}x\leq\ovr{b},\ -\omace{A}x\leq-\uvr{b},\ x\geq 0,
\end{align*}
see \cite{Fie2006,Mach1970,Roh1976}.
On the other hand, computing the worst case optimal value $\onum{f}$ is NP-hard \cite{GabMur2008,GabMur2010,Roh1997}. Some computationally cheap bounds were proposed in \cite{Hla2014d,MohGen2019}. An explicit but exponential formula for $\onum{f}$ is due to \cite{Fie2006,Roh1984}. It reduces the problem into $2^m$ instances of real-valued LP problems:
\begin{align}\label{thmOvrTypaOnumf}
\onum{f}(\imace{A},\ivr{b},\ivr{c})&= \max_{s\in\{\pm1\}^m} 
  f(\Mid{A}-\diag(s)\Rad{A},\,
   \Mid{b}+\diag(s)\Rad{b},\,\ovr{c}),
\end{align}
where $\diag(s)$ denotes the diagonal matrix with entries $s_1,\dots,s_m$.

Notice that interval LP problem in other basic forms have different properties in general and the standard transformations need not be equivalent \cite{GarHla2019c}.
A unified method to approach the optimal value range problem was presented in \cite{Hla2009b}.

\paragraph{Structure of the paper.}
Section~\ref{sDw} introduces our sensitivity measure based on the worst-case optimal value derivative. For a nondegenerate probem, it can be easily calculated by an explicit formula. The degenerate case is discussed in Section~\ref{ssDegen}, where we present basic properties and propose an upper bound; Section~\ref{ssCompl} shows then that the problem is NP-hard. Another characterization of the sensitivity measure is given in Section~\ref{ssAlpha}. Since the measure is not invariant with respect to the scaling of the perturbation patterns, we also introduce its normalized version (Section~\ref{ssNorm}). Section~\ref{ssOtherLP} discusses LP problems in other canonical forms; it turns out that while equality constraints are intractable, inequalities could be tractable, but the conjecture remains open. Section~\ref{sEx} illustrates the sensitivity measure on various well and ill-conditioned problems; we also experiment with the benchmark data from Netlib Repository.

\section{Worst-case optimal value derivatives}\label{sDw}

Directional derivatives of the optimal value function in some given direction were already discussed; see, e.g., Freund \cite{Fre1985} or Gal and Greenberg \cite[chap.~3]{GalGre1997}. 
Our approach is different as we consider the worst case derivative.

\paragraph{Worst case type derivative.}
Let $\Rad{A}\geq0$, $\Rad{b}\geq0$ and $\Rad{c}\geq0$ be given and define an interval matrix and interval vectors
\begin{align*}
\imace{A}_\alpha &:=[A-\alpha\Rad{A},\,A+\alpha\Rad{A}],\\ 
\ivr{b}_\alpha   &:=[b-\alpha\Rad{b},\,b+\alpha\Rad{b}],\\ 
\ivr{c}_\alpha   &:=[c-\alpha\Rad{c},\,c+\alpha\Rad{c}],
\end{align*}
depending on a parameter $\alpha\geq0$.

We introduce the derivative of the worst optimal value with respect to $\alpha$, that is, the derivative $d_w(A,b,c)$ of the function $\alpha\mapsto \onum{f}(\imace{A}_\alpha,\ivr{b}_\alpha,\ivr{c}_\alpha)$. For $\alpha=0$, we have $\onum{f}(\imace{A}_\alpha,\ivr{b}_\alpha,\ivr{c}_\alpha)=f(A,b,c)$. So by definition of the derivative,
\begin{align*}
d_w(A,b,c,\Rad{A},\Rad{b},\Rad{c})=\lim_{\alpha\to0^+} 
 \frac{\onum{f}(\imace{A}_\alpha,\ivr{b}_\alpha,\ivr{c}_\alpha)-f(A,b,c)}{\alpha}.
\end{align*}
We will write simply $d_w$ or $d_w(A,b,c)$ instead of $d_w(A,b,c,\Rad{A},\Rad{b},\Rad{c})$ when there is no ambiguity.

\paragraph{The formula for $d_w$.}
To derive a closed form formula for $d_w(A,b,c)$, we first remind auxiliary results from Rohn \cite{Roh1989b}, see also Tigan and Stancu-Minasian \cite{TigSta2000}.

\begin{lemma}\label{lmmRohn}
If \nref{lpGen} has a unique nondegenerate optimal solution $x^*$ corresponding to a basis~$B$, and if $y^*$ is a dual optimal solution, then for a sufficiently small $\alpha>0$ and for every $A'\in\imace{A}_\alpha$, $b'\in\ivr{b}_\alpha$ and $c'\in\ivr{c}_\alpha$
\begin{align}\label{eqRohnPertF}
f(A',b',c')
=f(A,b,c) +  y^*{}^T(A-A')x^*
  + y^*{}^T(b'-b)+ x^*{}^T(c'-c)
  +\mna{O}(\alpha^2).
\end{align}
\end{lemma}

\begin{lemma}\label{lmmRohnWorst}
If \nref{lpGen} has a unique nondegenerate optimal solution $x^*$ corresponding to a basis~$B$, and if $y^*$ is a dual optimal solution, then for sufficiently small $\alpha>0$
\begin{align*}
\onum{f}(\imace{A}_\alpha,\ivr{b}_\alpha,\ivr{c}_\alpha)
& = f(A,b,c) + \alpha |y^*|^T\Rad{A}x^* + \alpha |y^*|^T\Rad{b}
    +\alpha x^*{}^T\Rad{c} +\mna{O}(\alpha^2)\\
& = f(A,b,c) + \alpha |y^*|^T\Rad{A}_Bx^*_B + \alpha |y^*|^T\Rad{b}
    +\alpha (x^*_B)^T\Rad{c}_B +\mna{O}(\alpha^2).
\end{align*}
\end{lemma}

\begin{proposition}\label{propNondeg}
If \nref{lpGen} has the unique nondegenerate optimal solution $x^*$, and if $y^*$ is a dual optimal solution, then
\begin{align*}
d_w
= |y^*|^T\Rad{A}_Bx^*_B + |y^*|^T\Rad{b}+ (x^*_B)^T\Rad{c}_B.
\end{align*}
\end{proposition}

\begin{proof}
By Lemma~\ref{lmmRohnWorst},
\begin{align*}
\frac{\onum{f}(\imace{A}_\alpha,\ivr{b}_\alpha,\ivr{c}_\alpha)-f(A,b,c)}{\alpha}
= |y^*|^T\Rad{A}_Bx^*_B + |y^*|^T\Rad{b}+(x^*_B)^T\Rad{c}_B+\mna{O}(\alpha).
\end{align*}
By taking the limit as $\alpha\to0^+$, we get the statement.
\end{proof}

There are two natural choices for the perturbation patterns $\Rad{A}$, $\Rad{b}$ and $\Rad{c}$. For relative perturbations, we take the absolute value of the nominal values, that is, $\Rad{A}=|A|$, $\Rad{b}=|b|$ and $\Rad{c}=|c|$. In this case, we have
\begin{align*}
d_w(A,b,c,\Rad{A},\Rad{b},\Rad{c})
= |y^*|^T|A_B|x_B + |y^*|^T|b|+ (x^*_B)^T|c|.
\end{align*}
For absolute  perturbations, all entries of $\Rad{A}$, $\Rad{b}$ and $\Rad{c}$ consist of ones, that is, $\Rad{A}=ee^T$, $\Rad{b}=e$ and $\Rad{c}=e$. In this case, we have
\begin{align*}
d_w(A,b,c,\Rad{A},\Rad{b},\Rad{c})
&= |y^*|^Tee^Tx_B + |y^*|^Te+ (x^*_B)^Te\\
&=(|y^*|^Te+1)((x^*_B)^Te+1)-1.
\end{align*}

\subsection{Degenerate case}\label{ssDegen}

For a degenerate optimal solution, the computation of $d_w$ can be more difficult. First, we present a variation of Lemma~\ref{lmmRohn}. In fact, Rohn \cite{Roh1989b} showed more: Equation \nref{eqRohnPertF} holds even for a possibly degenerate optimal solution $x^*$ provided $B$ is an optimal basis corresponding to $(A',b',c')$. This will help us to obtain upper bounds on $d_w(A,b,c)$ in the general case.

\begin{lemma}\label{lmmRohnAdap}
Let $x^*$ be a basic optimal solution to \nref{lpGen} corresponding to a basis~$B$, and let $y^*$ be the corresponding dual basic optimal solution. For every $A'\in\imace{A}_\alpha$, $b'\in\ivr{b}_\alpha$ and $c'\in\ivr{c}_\alpha$ having $B$ as an optimal basis,
\begin{align}\label{eqRohnPertFadap}
f(A',b',c')
=f(A,b,c) + y^*{}^T(A-A')x^*
  + y^*{}^T(b'-b)+ x^*{}^T(c'-c)
  +\mna{O}(\alpha^2).
\end{align}
\end{lemma}

We will assume some kind of regularity in order that a small perturbation does not make it primary or dually infeasible: 
\begin{itemize}
\item
Matrix $A$ has linearly independent rows and there is a primal feasible $x^0$ such that $x^0>0$.
\item
The dual feasible set has nonempty interior.
\end{itemize}
These assumptions guarantee that for a small enough perturbation an optimal solution exists.

\begin{proposition}\label{propDwUpperBound}
We have
\begin{align*}
d_w \leq \max_{B\in\mna{B}}\ d_w(B),
\end{align*}
where
\begin{align*}
d_w(B) = |y^*(B)|^T\Rad{A}x^*(B) + |y^*(B)|^T\Rad{b} + x^*(B)^T\Rad{c},
\end{align*}
$\mna{B}$ is the set of all optimal bases, $x^*(B)$ is the optimal solution and $y^*(B)$ the dual optimal solution corresponding to basis~$B$.
\end{proposition}

\begin{proof}
Let $\alpha>0$ be small enough. By Lemma~\ref{lmmRohnAdap}, for every $A'\in\imace{A}_\alpha$, $b'\in\ivr{b}_\alpha$ and $c'\in\ivr{c}_\alpha$ there is a basis $B\in\mna{B}$ that is optimal for both $(A,b,c)$ and $(A',b',c')$, and 
\begin{align*}
f(A',b',c')
&=f(A,b,c) + y^*(B)^T(A-A')x^*(B)
  + y^*(B)^T(b'-b)\\
&\quad + x^*(B)^T(c'-c) + \mna{O}(\alpha^2).
\end{align*}
Then we have
\begin{align*}
f(A',b',c')
\leq f(A,b,c) + \alpha |y^*(B)|^T\Rad{A}x^*(B) + \alpha |y^*(B)|^T\Rad{b}
      + \alpha x^*(B)^T\Rad{c} + \mna{O}(\alpha^2).
\end{align*}
Thus
\begin{align*}
\onum{f}(\imace{A}_\alpha,\ivr{b}_\alpha,\ivr{c}_\alpha)
\leq \max_{B\in\mna{B}}\big\{ 
& f(A,b,c) + \alpha |y^*(B)|^T\Rad{A}x^*(B)
  + \alpha |y^*(B)|^T\Rad{b}\\
&  + \alpha x^*(B)^T\Rad{c} + \mna{O}(\alpha^2)\big\},
\end{align*}
and the rest is analogous to Proposition~\ref{propNondeg}.
\end{proof}

\begin{proposition}\label{propDwAttain}
We have 
$
d_w =  d_w(B)
$ 
for certain $B\in\mna{B}$.
\end{proposition}

\begin{proof}
First, by \nref{thmOvrTypaOnumf}, the value $\onum{f}(\imace{A}_\alpha,\ivr{b}_\alpha,\ivr{c}_\alpha)$ is attained at a realization $(A-\alpha\diag(s)\Rad{A},b+\alpha\diag(s)\Rad{b},c+\alpha\Rad{c})$ and certain $s\in\{\pm1\}^m$. Next, recall that a basis $B$ is optimal if and only if $A_B$ is nonsingular, $A_B^{-1}b\geq0$ and $c_N^T-c_B^TA_B^{-1}A_N^{}\geq0^T$. Now, fix $B\in\mna{B}$ and $s\in\{\pm1\}^m$. Then the set of all values of $\alpha\geq0$, for which $B$ is an optimal basis for the scenario $(A-\alpha\diag(s)\Rad{A},b+\alpha\diag(s)\Rad{b},c+\alpha\Rad{c})$ is characterized by polynomial inequalities and is therefore formed by a union of finitely many intervals. Hence also the set of all values of $\alpha\geq0$, for which $B$ is an optimal basis for which $\onum{f}(\imace{A}_\alpha,\ivr{b}_\alpha,\ivr{c}_\alpha)$ is attained, is a union of finitely many intervals. This means that there is a sufficiently small $\alpha^0>0$ such that for all $\alpha\in[0,\alpha^0]$, the worst case optimal value $\onum{f}(\imace{A}_\alpha,\ivr{b}_\alpha,\ivr{c}_\alpha)$ is attained for basis $B$ and $s\in\{\pm1\}^m$. Vector $s$ is the sign vector of a corresponding dual basic optimal solution $y^*(B)$ \cite{Mra1998}, that is, $s=\sgn(y^*(B))$.
Now, by Lemma~\ref{lmmRohnAdap}
\begin{align*}
\onum{f}(\imace{A}_\alpha,\ivr{b}_\alpha,\ivr{c}_\alpha)
&=f\big(A-\alpha\diag(s)\Rad{A},b+\alpha\diag(s)\Rad{b},c+\alpha\Rad{c}\big)\\
&= f(A,b,c) + \alpha y^*(B)^T\diag(s)\Rad{A}x^*(B) + \alpha y^*(B)^T\diag(s)\Rad{b}\\
&\quad  + \alpha  x^*(B)^T\Rad{c} + \mna{O}(\alpha^2).
\end{align*}
We have $y^*(B)^T\diag(s)\Rad{A}x^*(B)=|y^*(B)|^T\Rad{A}x^*(B)$ and similarly for the other terms. Thus
\begin{align*}
\onum{f}(\imace{A}_\alpha,\ivr{b}_\alpha,\ivr{c}_\alpha)
&= f(A,b,c) + \alpha |y^*(B)|^T\Rad{A}x^*(B) + \alpha |y^*(B)|^T\Rad{b}\\
&\quad  + \alpha x^*(B)^T\Rad{c} + \mna{O}(\alpha^2),
\end{align*}
from which $d_w =  d_w(B)$.
\end{proof}

The drawback of this formula is that we have to calculate all optimal bases, the number of which can possibly be high (exponential w.r.t.\ the dimension).
Moreover, it is an open problem how to check that $d_w =  d_w(B)$ for a given basis~$B$. There is no characterization known and also its computational complexity.

\subsection{Computational complexity}\label{ssCompl}

We show that it is intractable to compute not only $d_w$, but also its upper bound.

\begin{proposition}\label{propDwNP}
It is NP-hard to check if $d_w\geq1$.
\end{proposition}

\begin{proof}
By \cite[Thm.~2.3]{Fie2006}, it is NP-complete to check solvability of 
\begin{align}\label{ineqEqLemmaNP}
-e\leq My\leq e,\ e^T|y|\geq1
\end{align}
in the set of non-negative positive definite rational matrices~$M$. We will construct a reduction to this problem as follows. Put $A=(M\mid -M)$, $b=0$, $c=e$, $\Rad{A}=0$, $\Rad{b}=e$ and $\Rad{c}=0$. Thus the LP problem reads
$$
\min\ e^Tx^1+e^Tx^2 \st Mx^1-Mx^2=0,\ x^1,x^2\geq0,
$$
and its dual takes the form
\begin{align}\label{lpPropDwNPdual}
\max\ 0^Ty \st My\leq e,\ -My\leq e.
\end{align}
The dual feasible set is nonempty and bounded. Moreover, each dual feasible solution is optimal and $f(A,b,c)=0$. By \cite{Hla2012a}, we have
\begin{align*}
\onum{f}(\imace{A}_\alpha,\ivr{b}_\alpha,\ivr{c}_\alpha)
=\alpha f^*,
\end{align*}
where
\begin{align*}
f^* = \max\ e^T|y| \st -e\leq My\leq e.
\end{align*}
Hence $d_w=f^*$, and we have $d_w\geq1$ if and only if the system \nref{ineqEqLemmaNP} is feasible.
\end{proof}

\begin{proposition}
It is NP-hard to check $\max_{B\in\mna{B}}\, d_w(B)\geq1$.
\end{proposition}

\begin{proof}
We proceed in the same way as in the proof of Proposition~\ref{propDwNP}. Since for the dual LP problem \nref{lpPropDwNPdual} each feasible basis is optimal, we get $d_w=\max_{B\in\mna{B}}\, d_w(B)$.
\end{proof}

There are, however, some cases when the upper bound is tractable.

\begin{remark}
Suppose that the dual optimal solution $y^*$ is unique. Then for every $B\in\mna{B}$ we have
\begin{align*}
d_w(B) = |y^*|^T\Rad{A}x^*(B) + |y^*|^T\Rad{b} + x^*(B)^T\Rad{c}.
\end{align*}
The function 
\begin{align*}
h(x) = |y^*|^T\Rad{A}x + |y^*|^T\Rad{b} + x^T\Rad{c}
\end{align*}
is linear in $x$, and on the set of optimal solutions it attains its maximum at some basic solution, so the maximum value is $\max_{B\in\mna{B}}\,d_w(B)$. Therefore $\max_{B\in\mna{B}}\,d_w(B)$ can be computed in polynomial time by solving the LP problem
\begin{align*}
\max\ h(x) \st Ax=b,\ x\geq0,\ A^Ty\leq c,\ c^Tx=b^Ty.
\end{align*}

Another tractable case is when $\Rad{A}=0$ and $\Rad{b}=0$. In this case, $d_w(B) = x^*(B)^T\Rad{c}$. Thus, $\max_{B\in\mna{B}}\,d_w(B)$ can again be computed efficiently max maximizing a linear function on the set of optimal solutions, that is,
\begin{align*}
\max\ x^T\Rad{c} \st Ax=b,\ x\geq0,\ A^Ty\leq c,\ c^Tx=b^Ty.
\end{align*}
\end{remark}

\subsection{Another approach}\label{ssAlpha}
In view of \nref{thmOvrTypaOnumf}, there is $s\in\{\pm1\}^m$ such that
\begin{align}\label{dwFixS}
d_w=\lim_{\alpha\to0^+} 
 \frac{f(A_\alpha,b_\alpha,c_\alpha)-f(A,b,c)}{\alpha},
\end{align}
where
\begin{align*}
A_\alpha=A-\alpha\diag(s)\Rad{A},\ \ 
b_\alpha=b+\alpha\diag(s)\Rad{b},\ \ 
c_\alpha=c+\alpha\Rad{c}.
\end{align*}
This does not provide us with an effective way of computation of $d_w$, however, it shows the structure and direction of the worst case perturbation. 

Further, it suggests a possible (yet exponential) method of calculation $d_w$. Let $s\in\{\pm1\}^m$ and let us focus on computation the right-hand side value of~\nref{dwFixS}. Let $B$ be an optimal basis corresponding to $(A,b,c)$. Since $A_\alpha$, $b_\alpha$ and $c_\alpha$ depend linearly on the parameter $\alpha$, the entries in the simplex table depend on $\alpha$ as a polynomial fraction: recall that the right-hand side vector is $(A_\alpha)_B^{-1}b_\alpha$, the objective row $(c_\alpha)_N^T-(c_\alpha)_B^T(A_\alpha)_B^{-1}(A_\alpha)_N^{}$ and the table itself $(A_\alpha)_B^{-1}(A_\alpha)_N^{}$. Considering the parameter $\alpha$ as an infinitesimal value, we can virtually perform the simplex table transformation. Using the standard pivot selection rules, we move to the neighboring optimal basis. We proceed until we arrive at a basis $B^s$ that is optimal for the infinitesimal~$\alpha$. In view of Proposition~\ref{propDwAttain}, $d_w$ is attained for this basis and a certain $s\in\{\pm1\}^m$, that is, $d_w =  \max_{s\in\{\pm1\}^m} d_w(B^s)$.

Regarding the simplex table, there is no need to perform the actual transformations of the table; we just need to find the next basis. Further, in order to find the pivot and the subsequent basis, we need to compare the table elements and check if they are positive, for $\alpha>0$ sufficiently small. This can be achieved by calculating the derivative, or a higher order derivative in case it is zero.

Since the simplex method is known to take exponentially many steps in the worst case, the direct implementation would not yield a provably polynomial method for the computation of $B^s$. Indeed, it is an open question what is the true complexity of computing $B^s$ for a fixed $s\in\{\pm1\}^m$.


\subsection{A normalized version}\label{ssNorm}
Since the worst case optimal value derivative is not invariant with respect to  scaling of $(\Rad{A},\Rad{b},\Rad{c})$, we introduce an alternative derivative, in the definition of which $(\Rad{A},\Rad{b},\Rad{c})$ is inherently normalized. 
It is basically the maximal directional derivative of the optimal value in the space of $A,b,c$. More concretely, we consider the function  $(A,b,c)\mapsto f(A,b,c)$ and $d_r(A,b,c)$ denotes the maximum directional derivative on an interval box $\mathfrak{B}$, taking into account the scaling factors $\Rad{A}$, $\Rad{b}$ and $\Rad{c}$. The box reads 
$\mathfrak{B}=(\imace{A}_\tau,\ivr{b}_\tau,\ivr{c}_\tau)$,
where $\tau=1/\|(\Rad{A},\Rad{b},\Rad{c})\|_F$. 
Formally, we define
\begin{align*}
&d_r(A,b,c,\Rad{A},\Rad{b},\Rad{c})\\
&=\max_{(A',b',c')\in \mathfrak{B}}\ 
\lim_{\alpha\to0^+} 
 \frac{\onum{f}(A+\alpha(A'-A),b+\alpha(b'-b),c+\alpha(c'-c))-f(A,b,c)}{\alpha}.
\end{align*}

\begin{proposition}
We have:
\begin{enumerate}[(1)]
\item
$\displaystyle
d_r(A,b,c,\Rad{A},\Rad{b},\Rad{c})
 = \frac{1}{\|(\Rad{A},\Rad{b},\Rad{c})\|_F} d_w(A,b,c,\Rad{A},\Rad{b},\Rad{c}),
$
\item
$
d_w(\beta {A},\beta {b},\beta {c},\gamma\Rad{A},\gamma\Rad{b},\gamma\Rad{c})
= \gamma \cdot d_w(A,b,c,\Rad{A},\Rad{b},\Rad{c}),\ \forall \beta,\gamma>0,
$
\item
$
d_r(\beta {A},\beta {b},\beta {c},\gamma\Rad{A},\gamma\Rad{b},\gamma\Rad{c})
= d_r(A,b,c,\Rad{A},\Rad{b},\Rad{c}),\ \forall \beta,\gamma>0.
$
\end{enumerate}
\end{proposition}

\begin{proof}\mbox{}
\begin{enumerate}[(1)]
\item
The largest increase of the optimal value function on $B$ is $\onum{f}(\imace{A}_\tau,\ivr{b}_\tau,\ivr{c}_\tau)$. Therefore,
\begin{align*}
d_r=\lim_{\alpha\to0^+} 
  \frac{\onum{f}(\imace{A}_{\alpha\tau},\ivr{b}_{\alpha\tau},\ivr{c}_{\alpha\tau})
    -f(A,b,c)}{\alpha}
 =\tau \cdot d_w.
\end{align*}
\item
By definition,
\begin{align*}
d_w(\beta A,\beta b,\beta c,\Rad{A},\Rad{b},\Rad{c})
&=\lim_{\alpha\to0^+} 
  \frac{\onum{f}(\beta \imace{A}_{\alpha/\beta},\beta \ivr{b}_{\alpha/\beta},
 \beta \ivr{c}_{\alpha/\beta})-\beta\cdot f(A,b,c)}{\alpha}\\
&=\lim_{\alpha\to0^+} 
  \frac{\onum{f}(\imace{A}_{\alpha/\beta},\ivr{b}_{\alpha/\beta},
  \ivr{c}_{\alpha/\beta})- f(A,b,c)}{\alpha/\beta}\\
&= d_w(A,b,c,\Rad{A},\Rad{b},\Rad{c}).
\end{align*}
Similarly,
\begin{align*}
d_w(A,b,c,\gamma\Rad{A},\gamma\Rad{b},\gamma\Rad{c})
&=\lim_{\alpha\to0^+} 
  \frac{\onum{f}(\imace{A}_{\gamma\alpha},\ivr{b}_{\gamma\alpha},
   \ivr{c}_{\gamma\alpha}) - f(A,b,c)}{\alpha}\\
&=\lim_{\alpha\to0^+}  \gamma
  \frac{\onum{f}(\imace{A}_{\gamma\alpha},\ivr{b}_{\gamma\alpha},
   \ivr{c}_{\gamma\alpha}) - f(A,b,c)}{\gamma\alpha}\\
&= \gamma\cdot d_w(A,b,c,\Rad{A},\Rad{b},\Rad{c}).
\end{align*}
\item
We have
\begin{align*}
d_r(\beta {A},\beta {b},\beta {c},\gamma\Rad{A},\gamma\Rad{b},\gamma\Rad{c})
&= \frac{d_w(\beta {A},\beta {b},\beta {c},\gamma\Rad{A},\gamma\Rad{b},
      \gamma\Rad{c})}{\|\gamma(\Rad{A},\Rad{b},\Rad{c})\|_F}\\
&= \frac{\gamma \cdot d_w(A,b,c,\Rad{A},\Rad{b},\Rad{c})}
   {\gamma\cdot \|(\Rad{A},\Rad{b},\Rad{c})\|_F}\\
&= d_r(A,b,c,\Rad{A},\Rad{b},\Rad{c}).
\qedhere
\end{align*}
\end{enumerate}
\end{proof}

\paragraph{Special cases.}
In the special case when we allow perturbation of one coefficient only, we obtain Proposition~\ref{propNondeg} the standard differentiability results (up to the absolute value); see Gal and Greenberg \cite[chap.~3]{GalGre1997}. That is, the primal and dual optimal solution entries are the derivatives of the optimal value function with respect to the objective and right-hand side vector entries, respectively.

\begin{corollary}
Let $x^*$ be the unique nondegenerate optimal solution and $y^*$ the dual optimal solution. 
\begin{enumerate}[(1)]
\item
If $\Rad{A}=0$, $\Rad{b}=0$ and $\Rad{c}=e_j$, then $d_w=d_r=x^*_j$.
\item
If $\Rad{A}=0$, $\Rad{b}=e_i$ and $\Rad{c}=0$, then $d_w=d_r=|y^*_i|$.
\item
If $\Rad{A}=e_i^{}e_j^T$, $\Rad{b}=0$ and $\Rad{c}=0$, then $d_w=d_r=|y^*_ix^*_j|$.
\end{enumerate}
\end{corollary}

Now, we consider some other special situations, extending the above corollary. Let $\Rad{A}=0$ and $\Rad{b}=0$ and suppose that the optimal solution $x^*$ is not dual degenerate. By Propositions~\ref{propDwUpperBound} and~\ref{propDwAttain}, we get $d_w=x^*(B)^T\Rad{c}$ for some basis~$B$. Since $x^*=x^*(B)$ for every $B\in\mna{B}$, we simply have $d_w=(x^*)^T\Rad{c}$. 
Similarly, if $\Rad{A}=0$ and $\Rad{c}=0$ and the LP problem is not primal degenerate, then $d_w=|y^*|^T\Rad{b}$.  

Eventually, let  $\Rad{A}=0$, $\Rad{b}=0$ and $\Rad{c}=c\geq0$. Then 
\begin{align}\label{ofSpecCaseCnonneg}
\onum{f}(\imace{A}_\alpha,\ivr{b}_\alpha,\ivr{c}_\alpha)
=f(A,b,c+\alpha c)
=(1+\alpha)f(A,b,c),
\end{align}
from which $d_w=f(A,b,c)$.

\subsection{Other LP forms}\label{ssOtherLP}

Besides \nref{lpGen}, the feasible set of an LP problem can be formulated by a system of linear inequalities or by some other standard form. Basically, we can easily transform the other types to the form of \nref{lpGen} and apply our methods. More concretely, an LP problem in the form
\begin{align}\label{lpTypC}
\min\ c^Tx \st Ax\leq b,\ x\geq0
\end{align}
is transformed to 
\begin{align*}
\min\ c^Tx \st Ax+I_mx'=b,\ x,x'\geq0.
\end{align*}
The perturbation patterns are set to $\Rad{(A\mid I_m)}:=(\Rad{A}\mid 0)$ in order that the identity matrix is considered as fixed.

An LP problem in the form
\begin{align}\label{lpTypB}
\min\ c^Tx \st Ax\leq b
\end{align}
is transformed to 
\begin{align*}
\min\ c^Tx^1-c^Tx^2 \st Ax^1-Ax^2+I_mx'=b,\ x^1,x^2,x'\geq0.
\end{align*}
The perturbation pattern for the constraint matrix is set to $\Rad{(A\mid -A \mid I_m)}:=(\Rad{A}\mid\Rad{A}\mid 0)$ and analogously for the objective function. To ensure that the problems are equivalent, we should consider structured perturbations of the data taking into account correlations between the entries of $(A\mid -A \mid I_m)$. Such problems are very hard in general! Fortunately, we need not consider  structured perturbations and take into account such correlations since it was proved by Garajov\'{a} et al.~\cite{GarHla2019c} that the worst case optimal value $\onum{f}$ of an interval LP problem is not changed under this transformation (in contrast to other characteristics). This property is also valid when imposing nonnegative variables in linear equations, so we can use the standard transformation and formulate each LP problem as~\nref{lpGen}.

Nevertheless, the transformation to the canonical form need not be the best approach from the computational point of view. Consider, for example, the situation with $\Rad{A}=0$, $\Rad{c}=0$ and $\Rad{b}=b\geq0$ (which makes a perfect sense in network flow problems, among others). Then analogously to \nref{ofSpecCaseCnonneg} we get $d_w=f(A,b,c)$. However, the transformation to the form \nref{lpGen} induces the equation constraints, which are hard to deal with from our perspective.

Moreover, the LP form \nref{lpTypC} has other possible advantages. Since
\begin{align*}
\onum{f}(\imace{A}_\alpha,\ivr{b}_\alpha,\ivr{c}_\alpha)
=f(A+\alpha\Rad{A},b-\alpha\Rad{b},c+\alpha\Rad{c}),
\end{align*}
we have
\begin{align*}
d_w=\lim_{\alpha\to0^+} 
 \frac{f(A+\alpha\Rad{A},b-\alpha\Rad{b},c+\alpha\Rad{c})-f(A,b,c)}{\alpha},
\end{align*}
and we may proceed in a similar way as in Section~\ref{ssAlpha}. Even though the method was rather cumbersome, we do not need to process all sign vectors $s\in\{\pm1\}^m$. Again, it is open if a polynomial time method can be constructed. 

Regarding an LP problem in the form \nref{lpTypC}, a similar commentary can be given. By \cite{GarHla2019c},
\begin{align*}
\onum{f}(\imace{A}_\alpha,\ivr{b}_\alpha,\ivr{c}_\alpha)
={}&\max\ (c+\alpha\Rad{c})x^1-(c-\alpha\Rad{c})^Tx^2 
  \\&\stl (A+\alpha\Rad{A})x^1-(A-\alpha\Rad{A})x^2\leq b-\alpha\Rad{b},\ x^1,x^2\geq0.
\end{align*}
Therefore, this case can be handled in the same way as the form~\nref{lpTypC}.

\section{Examples}\label{sEx}

\begin{example}\label{exNestab}
Consider the LP problem \nref{lpGen} with
\begin{align*}
A=\begin{pmatrix}5& -7& 1\\7& -10& 1\end{pmatrix},\ \ 
b=\begin{pmatrix}1\\0\end{pmatrix},\ \ 
c=\begin{pmatrix}12\\-17\\\approx2\end{pmatrix}.
\end{align*}
The dual LP problem is illustrated in Figure~\ref{figExNestab}.
Let us set for the perturbation patterns as $\Rad{A}=|A|$, $\Rad{b}=|b|$ and $\Rad{c}=|c|$. 
The entry $c_3$ is approximately around~$2$, and we will consider three cases: $c_3=1.5$, $c_3=2.5$ and $c_3=2$. 

For $c_3=1.5$, we compute $d_w = 29.556$ and $d_r = 1.1494$, 
and for $c_3=2.5$, we compute $d_w = 479$ and $d_r = 18.571$. 
We see that the latter is much more unstable and sensitive to perturbations.
\begin{figure}[t]
\begin{center}
\psset{unit=6.5ex,arrowscale=1.5}
\begin{pspicture}(-3.5,-0.4)(5.1,4.3)
\psaxes[ticksize=2pt,labels=all,ticks=all]
{->}(0,0)(-3.2,-0.2)(2.9,4)
\newgray{mygray}{0.85}
\pspolygon[fillstyle=solid,fillcolor=mygray,linecolor=mygray,linewidth=0pt]
(-3,3.96)(1,1)(-3,3.7)
\psline[](-3,3.96)(1,1)(-3,3.7)
\psline[linestyle=dashed](-1.4,3.9)(2.3,0.2)
\psline[linestyle=dashed](-1.9,3.9)(1.8,0.2)
\psline[linestyle=dashed](-2.4,3.9)(1.3,0.2)
\psline[](3,2.8)(3,4)
\psline[]{->}(3,3.4)(3.5,3.4)
\uput[0](3.5,3.4){$\max$}

\uput[45](1.5,1){$y_1+y_2\leq c_3$}
\uput[-90](2.9,0){$y_1$}
\uput[180](0,4){$y_2$}
\uput[-135](-0.01,-0.01){$0$}
\end{pspicture}
\caption{(Example~\ref{exNestab}) The dual LP problem with the inequality $y_1+y_2\leq c_3$ having the values of $c_3\in\{1.5,\,2,\,2.5\}$.\label{figExNestab}}
\end{center}
\end{figure}
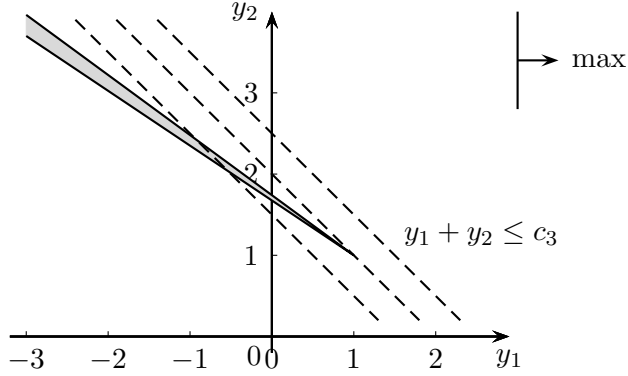

Now, let $c_3=2$, for which the problem is degenerate. There are two optimal bases, $B_1=\{1,2\}$ and $B_2=\{2,3\}$. The first one corresponds to the basic optimal solution $x^{B_1}=(10,7,0)^T$ and dual optimal solution $y^*=(1,1)^T$. We compute $d_w(B_1) = 479$ and $d_r(B_1) = 18.603$. The second basis is associated with the basic optimal solution $x^{B_2}=(0,\frac{1}{3},\frac{10}{3})^T$, and the dual optimal solution is the same, i.e., $y^*$. The corresponding derivatives are $d_w(B_2) = 25.667$ and $d_r(B_2) = 0.99681$. So by Proposition~\ref{propDwUpperBound}, the worst-case derivative satisfies $d_w\leq d_w(B_1) = 479$ and $d_r\leq d_r(B_1) = 18.603$.
These upper bounds are not tight. In view of \nref{thmOvrTypaOnumf}, it is not hard to compute analytically the exact values of the derivatives. The worst case optimal value $\onum{f}(\imace{A}_\alpha,\ivr{b}_\alpha,\ivr{c}_\alpha)$ is attained for basis $B_2$ and for the scenario 
\begin{align*}
A'=A-\alpha|A|,\ \ 
b'=b+\alpha|b|,\ \ 
c'=c+\alpha|c|.
\end{align*}
The corresponding optimal solution is $x(\alpha)=\big(0,\frac{1}{3},\frac{10(1+\alpha)}{3(1-\alpha)}\big)^T$, and the optimal value reads
$
\frac{3+74\alpha+3\alpha^2}{3(1-\alpha)}.
$
The derivative has the value of
$$
d_w
=\lim_{\alpha\to0^+} \frac{\frac{3+74\alpha+3\alpha^2}{3(1-\alpha)}-1}{\alpha}
=\lim_{\alpha\to0^+} \frac{77+3\alpha}{3(1-\alpha)}=\frac{77}{3}.
$$
Therefore $d_w=\frac{77}{3}\approx 25.667$ and $d_r\approx 0.99681$. We see that $d_w=d_w(B_2)$ and $d_r=d_r(B_2)$ in correspondence with Proposition~\ref{propDwAttain}.
\end{example}

\begin{example}\label{exCube}
Consider the example of maximizing $L_1$ norm on a unite $n$-dimensional hypercube
$$
\max\ e^Tx \st -e\leq x\leq e.
$$
If we put $\Rad{A}=|A|$, $\Rad{b}=|b|$ and $\Rad{c}=|c|$, then
$$
d_w=3n, \quad
d_r=\frac{3}{5}\sqrt{5n}.
$$
For $\Rad{A}=ee^T$, $\Rad{b}=e$ and $\Rad{c}=e$ we get
$$
d_w=n(n+2), \quad
d_r=\frac{n(n+2)}{\sqrt{n(2n+3)}}.
$$
This confirms that the problem is stable -- the directional derivative $d_r$ grows as square root and linearly in the dimension, respectively.

Consider now an example in which the hypercube is subject to a transformation by a matrix $A\in\R^{n\times n}$. Thus the LP problem draws
\begin{align}\label{lpCubeA}
\max\ c^Tx \st -e\leq Ax\leq e.
\end{align}
The objective vector $c$ is generated as a random vector with entries uniformly in $[-1,1]$. For matrix $A$ we consider four cases:
\begin{itemize}
\item
the identity matrix $I_n$,
\item
a random $n$-by-$n$ matrix with entries generated uniformly in $[-1,1]$,
\item
the Vandermonde matrix
$$
\begin{pmatrix}
v_1^{n-1}&\dots&v_1&1\\
v_2^{n-1}&\dots&v_2&1\\
\vdots &&\vdots&\vdots\\
v_n^{n-1}&\dots&v_n&1\\
\end{pmatrix}
$$
generated by vector $v=(v_1,\dots,v_n)=\big(\frac{1}{n},\frac{2}{n},\dots,\frac{n}{n}\big)$.
\item
the Hilbert matrix $H$ with entries $H_{ij}=\frac{1}{i+j-1}$.
\end{itemize}
The perturbation patterns are set up as $\Rad{A}=|A|$, $\Rad{b}=|b|$ and $\Rad{c}=|c|$. Table~\ref{tabCube} displays the results computed in \textsf{MATLAB R2017b}.

\begin{table}[t]
\caption{(Example~\ref{exCube}) Derivatives for the LP problem \nref{lpCubeA} with identity, random, Vandermonde and Hilbert matrices in the constraints and random objective functions.\label{tabCube}}
\begin{center}
\begin{tabular}{@{}ccccccccc@{}}
 \toprule
$n$ & \multicolumn{2}{c}{$A=I_n$}& \multicolumn{2}{c}{random}
    & \multicolumn{2}{c}{Vandermonde}& \multicolumn{2}{c}{Hilbert}\\
\cmidrule(lr){2-3}\cmidrule(lr){4-5}\cmidrule(lr){6-7}\cmidrule(lr){8-9}
& $d_w$ & $d_r$ & $d_w$ & $d_r$ & $d_w$ & $d_r$  & $d_w$ & $d_r$\\
\midrule 
 2& 2.545 & 0.876 & 3.953 & 1.338 & 23.35 & 8.287 & 144.3 & 52.13 \\ 
 3& 7.055 & 1.894 & 9.747 & 2.693 & 56.67 & 9.801 & $1.2\cdot10^4$ & 3468  \\ 
 4& 5.392 & 1.308 & 26.60 & 6.295 & 2034  & 78.66 & $1.4\cdot10^7$ & $3.7\cdot10^6$ \\ 
 5& 6.646 & 1.436 & 152.3 & 30.36 & $1.2\cdot10^5$ & 641.7 & $3.6\cdot10^{10}$ & $9.0\cdot10^9$ \\ 
 6& 9.640 & 1.887 & 106.0 & 17.75 & $2.9\cdot10^6$ & 1457 & $4.8\cdot10^{13}$ & $1.1\cdot10^{13}$ \\ 
 7& 13.76 & 2.468 & 27.09 & 3.777 & $8.0\cdot10^{7}$ & 3063 & $8.7\cdot10^{16}$ & $1.8\cdot10^{16}$ \\ 
 8& 9.683 & 1.666 & 92.10 & 12.20 & $5.5\cdot10^{7}$  & 136.3 & -- & -- \\ 
 9& 14.66 & 2.342 & 205.8 & 23.59 & $2.8\cdot10^{9}$  & 376.2 & -- & --  \\ 
10& 16.54 & 2.498 & 5251  & 575.2 & $1.9\cdot10^{10}$ & 120.5 & -- & --  \\ 
 \bottomrule
\end{tabular}
\end{center}
\end{table}

As expected, the first case with the identity matrix is very stable with low derivatives. Random matrices evince large but still mild derivatives. Vandermonde matrix causes very large derivatives and makes the problem highly unstable. The extremal case is the fourth example with Hilbert matrices. They are known to be ill-conditioned, which is reflected by huge derivatives -- indeed, for $n\geq8$ the LP solver is unable to compute the optimal solution.
\end{example}

\begin{example}\label{exBench}
Here we consider data from Netlib Repository (\url{http://www.netlib.org/}). The results are displayed in Table~\ref{tabBench}. Since the optimal solutions of all tested examples were either not unique or degenerate, we display the derivatives related to the resulting basis. For comparison, we slightly and randomly perturbed the data such that the problems become non-degenerate; more concretely, we perturb the data uniformly at random up to $0.005\%$ of their nominal values. The corresponding results are shown on odd rows.
\begin{table}[t]
\caption{(Example~\ref{exBench}) Derivatives for Netlib benchmark data. The odd rows are for perturbed problems. Herein, \quo{vars} denotes the number of variables and \quo{constr} the number of constraints.\label{tabBench}}
\begin{center}
\addtolength{\tabcolsep}{1pt}
\begin{tabular}{@{}cccccc@{}}
 \toprule
name & vars & constr & $d_w(B)$ & $d_r(B)$ & $f(A,b,c)$\\
\midrule 
AGG2 & 302 & 516 & $3.0\cdot 10^8$ & 97.8 & $-2.0\cdot 10^7$\\
               &&& $3.0\cdot 10^8$ & 97.75 & $-2.0\cdot 10^7$\\
BANDM &472 & 305 & 4686 & 4.128  & $-158.6$ \\
               &&& 7584 & 6.707  & $-78.44$ \\
BOEING1&384& 440 & 3644 & 0.2683 & $-335.2$\\
               &&& 3633 & 0.2674 & $-335.2$\\
CAPRI & 353& 271 & $1.5\cdot10^5$& 20.4 & 2690\\
               &&& $1.5\cdot10^5$& 20.39 & 2690\\
FIT1P &1677& 627 & 56010 & 4.96 & 9146 \\
               &&& 56010 & 4.961 & 9146 \\
FIT2P &13525&3000& $4.0\cdot10^5$& 34.73 & 68460\\
               &&& $4.0\cdot10^5$& 34.73 & 68460\\
GREENBEA &5405& 2392& $4.7\cdot10^9$& $3.4\cdot10^6$& $-7.3\cdot10^7$\\
                  &&& $8.5\cdot10^9$& $6.2\cdot10^6$& $-7.5\cdot10^7$\\
GREENBEB &5405& 2392& $1.7\cdot10^7$& 12650 & $-4.3\cdot10^6$\\
                  &&& $2.4\cdot10^7$& 17990 & $-4.3\cdot10^6$\\
ISREAL &142 & 174 & $4.0\cdot10^6$ & 3.892 & $-9.0\cdot10^5$\\
                &&& $4.0\cdot10^6$ & 3.892 & $-9.0\cdot10^5$\\
MAROS & 1443 & 846 & $2.4\cdot10^6$ &  15.7 & $-58060$\\
                 &&& $2.4\cdot10^6$ &  15.77 & $-58060$\\
PILOT & 3652 & 1441 & 13140 & 1.654 & $-557.5$\\
                  &&& 13130 & 1.652 & $-557.5$\\
SC105 & 103  & 105 & 1311  & 1.853 & $-52.2$\\
      &      &     & 1311  & 1.853 & $-52.2$\\
SCSD1 & 760  & 77  & 50.42 & 0.6369 & 8.667\\
                 &&& 66    & 0.8337 & 8.667\\
SCFXM2 & 914 & 660 & $4.1\cdot10^5$ & 70.8  & 36660\\
                 &&& $5.5\cdot10^6$ & 959.7 & 33110\\
SHIP04L & 2118 & 402 &$8.8\cdot10^6$ & 320.5 & $1.8\cdot10^6$\\
                  &&& $8.8\cdot10^6$ & 320.5 & $1.8\cdot10^6$\\
STANDATA & 1075 & 359 & 7756 & 2.684 & 1258\\
                    &&& 7936 & 2.746 & 1257\\
STOCFOR2 & 2031 & 2157 & $7.1\cdot10^5$ & 64.78 & $-39020$\\
                     &&& $7.1\cdot10^5$ & 64.78 & $-39020$\\
TRUSS & 8806 & 1000 & $1.7\cdot10^7$ & 9652 &  $4.6\cdot10^5$\\
                  &&& $1.7\cdot10^7$ & 9965 & $4.6\cdot10^5$\\
 \bottomrule
\end{tabular}
\end{center}
\end{table}

We can see that $d_w(B)$ and $d_r(B)$ are good approximations of $d_w$ and $d_r$. In many cases they are the same or differ only slightly. There are only few examples (BANDM, GREENBEA, GREENBEB, SCSD1 and SCFXM2), where the difference is more significant. For problem SCFXM2 in particular, the change is in the order of magnitude. This indicates that $d_w(B)$ and $d_r(B)$ can sometimes underestimate the true value.
\end{example}

\section{Conclusion}


We introduced the concept of a derivative of the optimal value function in linear programming. It is based on the worst case optimal value in the neighborhood when the data are inflated to intervals along some given perturbation patterns. 
The derivative shows how the optimal value can worsen when data are varying. It also gives some measure of sensitivity of an LP problem.
As an alternative, we introduced also a directional derivative, which is a normalized version of the above type derivative. Due to the normalization, it serves as a good sensitivity measure or a condition number of LP problems.

For a nondegenerate LP problem, the derivatives are easily computable from the primal and dual optimal solutions. For degenerate problems, the computation is more difficult; indeed, it is NP-hard for equality constrained problems. The computational complexity for inequality constrained problems remains unknown. Furthermore, we proposed an upper bound and also showed that the derivative is attained for a certain optimal  basis. Nevertheless, it is an open problem  how to recognize that the derivative is attained for a given basis and to determine its complexity.

The numerical examples and experiments that we caried out show that the derivatives reflect the sensitivities of LP problems to data variations. For degenerate problems, the computed basis gives a reasonable approximation of the true derivatives in most of the cases. On the other hand, it may sometimes underestimate, so one has to be careful when dealing with degenerate problems.

\subsubsection*{Acknowledgments.} 
The author was supported by the Czech Science Foundation Grant P403-22-11117S.


\bibliographystyle{abbrv}
\bibliography{drce_lp}

\end{document}